\documentclass[12pt, a4paper]{article}
\usepackage{latexsym}
\usepackage{xy}
\usepackage{amsfonts,amsmath,amssymb,amsthm}
\fontsize{11pt}{11pt}\selectfont

\xyoption{all}

\newcommand{\bcen}{\begin{center}}     \newcommand{\ecen}{\end{center}}
\newcommand{\bay}{\begin{array}}      \newcommand{\eay}{\end{array}}
\newcommand{\beq}{\begin{eqnarray*}}      \newcommand{\eeq}{\end{eqnarray*}}

\def\gl{\mathrm{gl.dim}}

\def\hl{\mbox{\rm hl}}
\def\hr{\mbox{\rm hr}}
\def\hw{\mbox{\rm hw}}

\def\Hom{\mathrm{Hom}}

\def\rad{\mathrm{rad}}
\def\op{\mathrm{op}}
\def\Ext{\mathrm{Ext}}

\def\End{\mathrm{End}}

\def\dim{\mathrm{dim}}
\def\mod{\mathrm{mod}}

\def\Im{\mathrm{Im}}
\def\Ker{\mathrm{Ker}}

\def\proj{\mathrm{proj}}

\def\RHom{\mathrm{RHom}}

\begin{document}

\newtheorem{theorem}{Theorem}
\newtheorem{proposition}{Proposition}
\newtheorem{lemma}{Lemma}
\newtheorem{corollary}{Corollary}
\newtheorem{remark}{Remark}
\newtheorem{example}{Example}
\newtheorem{definition}{Definition}
\newtheorem*{conjecture}{Conjecture}
\newtheorem*{question}{Question}

\title{\large\bf Brauer-Thrall type theorems for derived module categories}

\author{\large Chao Zhang$^1$ and Yang Han$^2$}

\date{\footnotesize 1. Department of Mathematics, Guizhou University, Guiyang 550025, P.R.
China \\ E-mail: zhangc@amss.ac.cn \\ 2. KLMM, ISS, AMSS, Chinese
Academy of Sciences, Beijing 100190, P.R. China.\\
E-mail: hany@iss.ac.cn}

\maketitle

\begin{abstract} The numerical invariants
(global) cohomological length, (global) cohomological width, and
(global) cohomological range of a complex (an algebra) are
introduced. Cohomological range leads to the concepts of derived
bounded algebra and strongly derived unbounded algebra naturally.
The first and second Brauer-Thrall type theorems for the bounded
derived category of a finite-dimensional algebra over an
algebraically closed field are obtained. The first Brauer-Thrall
type theorem says that derived bounded algebras are just derived
finite algebras. The second Brauer-Thrall type theorem says that an
algebra is either derived discrete or strongly derived unbounded,
but not both. Moreover, piecewise hereditary algebras and derived
discrete algebras are characterized as the algebras of finite global
cohomological width and the algebras of finite global cohomological
length respectively.
\end{abstract}

\medskip


\medskip


\bigskip

\section{\large Introduction}

Throughout this paper, $k$ is an algebraically closed field, all
algebras are connected basic finite-dimensional associative
$k$-algebras with identity, and all modules are finite-dimensional
right modules, unless stated otherwise. One of the main topics in
representation theory of algebras is to study the classification and
distribution of indecomposable modules. In this aspect two famous
problems are Brauer-Thrall conjectures I and II:

\medskip

{\bf Brauer-Thrall conjecture I.} {\it The algebras of bounded
representation type are of finite representation type.}

\medskip

{\bf Brauer-Thrall conjecture II.} {\it The algebras of unbounded
representation type are of strongly unbounded representation type.}

\medskip

\noindent Here, we say an algebra is {\it of finite representation
type} or {\it representation-finite} if there are only finitely many
isomorphism classes of indecomposable modules. An algebra is said to
be {\it of bounded representation type} if the dimensions of all
indecomposable modules have a common upper bound, and {\it of
unbounded representation type} otherwise. We say an algebra is {\it
of strongly unbounded representation type} if there are infinitely
many $d \in \mathbb{N}$ such that for each $d$, there exist
infinitely many isomorphism classes of indecomposable modules of
dimension $d$. Brauer-Thrall conjectures I and II were formulated by
Jans \cite{Jan57}. Brauer-Thrall conjecture I was proved for
finite-dimensional algebras over an arbitrary field by Roiter
\cite{Ro68}, and artin algebras by Auslander \cite{Aus74}.
Brauer-Thrall conjecture II was proved for finite-dimensional
algebras over an infinite perfect field by Nazarova and Roiter using
matrix method \cite{NR75, Ro78}, and an algebraically closed field
by Bautista using geometric method \cite{Bau85}. Refer to
\cite{Rin80} for more on Brauer-Thrall conjectures.

Since Happel \cite{Hap87,Hap88}, the bounded derived categories of
finite-dimensional algebras have been studied widely. The study of
the classification and distribution of indecomposable objects in the
bounded derived category of an algebra is still an important theme
in representation theory of algebras. It is natural to consider the
derived versions of Brauer-Thrall conjectures. For this, one needs
to find an invariant of a complex analogous to the dimension of a
module. On this topic, Vossieck is undoubtedly a pioneer. He
introduced and classified {\it derived discrete algebras}, i.e., the
algebras whose bounded derived categories admit only finitely many
isomorphism classes of indecomposable objects of arbitrarily given
cohomology dimension vector, in his elegant paper \cite{Vo01}. Since
a complex and its shifts are of different cohomology dimension
vectors, for a non derived discrete algebra, there are always
infinitely many $\mathbf{d} \in \mathbb{N}^{(\mathbb{Z})}$ such that
for each $\mathbf{d}$, there exist infinitely many isomorphism
classes of indecomposable objects of cohomology dimension vector
$\mathbf{d}$ in its bounded derived category. Nevertheless,
cohomology dimension vector is seemingly not a perfect invariant of
complexes in the context of derived versions of Brauer-Thrall
conjectures, because it is too fine to identify an indecomposable
complex with its shifts and cannot be used to define the derived
boundedness and strongly derived unboundedness of algebras.

In this paper, we shall introduce the cohomological range of a
bounded complex which is a numerical invariant under shifts and
isomorphisms. It leads to the concepts of derived bounded algebras
and strongly derived unbounded algebras naturally. We shall prove
the following two Brauer-Thrall type theorems for derived module
categories:

\medskip

{\bf Theorem I.} {\it Derived bounded algebras are just derived
finite algebras.}

\medskip

{\bf Theorem II.} {\it An algebra is either derived discrete or
strongly derived unbounded, but not both.}

\medskip

According to Theorem I and Theorem II, all algebras are divided into
three disjoint classes: derived finite algebras, derived discrete
but not derived finite algebras, and strongly derived unbounded
algebras. In particular, Theorem II excludes the existence of such
an algebra for which there are only (nonempty) finitely many $r \in
\mathbb{N}$ such that for each $r$, up to shift and isomorphism,
there exist infinitely many indecomposable objects of cohomological
range $r$ in its bounded derived category.

\medskip

The paper is organized as follows: in Section 2, we shall introduce
some numerical invariants of complexes (algebras) including (global)
cohomological length, (global) cohomological width, and (global)
cohomological range, and observe their behaviors under derived
equivalences. Global cohomological width provides an alternative
definition of strong global dimension on the level of bounded
derived category, and piecewise hereditary algebras are
characterized as the algebras of finite global cohomological width.
Furthermore, we shall prove Theorem I. In Section 3, we shall show
that strongly derived unboundedness is invariant under derived
equivalences, and observe its relation with cleaving functors.
Furthermore, we shall prove Theorem II for simply connected
algebras, gentle algebras, and finally all algebras by using
cleaving functors and covering theory. Moreover, derived discrete
algebras are characterized as the algebras of finite global
cohomological length.

\section{\large The first Brauer-Thrall type theorem}

\subsection{\normalsize Some numerical invariants of complexes and algebras}

Let $A$ be a (finite-dimensional) $k$-algebra. Denote by $\mod A$
the category of all (finite-dimensional) right $A$-modules, and by
$\proj A$ its full subcategory consisting of all finite-dimensional
projective right $A$-modules. Denote by $C(A)$ the category of all
complexes of finite-dimensional right $A$-modules, and by $C^b(A)$
and $C^{-,b}(A)$ its full subcategories consisting of all bounded
complexes and right bounded complexes with bounded cohomology
respectively. Denote by $C^b(\proj A)$ and $C^{-,b}(\proj A)$ the
full subcategories of $C^b(A)$ and $C^{-,b}(A)$ respectively
consisting of all complexes of finite-dimensional projective
modules. Denote by $K(A)$, $K^b(\proj A)$ and $K^{-,b}(\proj A)$ the
homotopy categories of $C(A)$, $C^b(\proj A)$ and $C^{-,b}(\proj A)$
respectively. Denote by $D^b(A)$ the bounded derived category of
$\mod A$. Moreover, $\dim := \dim_k$, the dimension of a $k$-vector
space.

\medskip

Now we introduce some numerical invariants of complexes and
algebras.

\begin{definition}{\rm The {\it cohomological length} of a complex $X^{\bullet} \in
D^b(A)$ is $$\hl(X^{\bullet}) := \max\{\dim H^i(X^{\bullet}) \; | \;
i \in \mathbb{Z}\},$$ and the {\it global cohomological length} of
$A$ is
$$\mbox{\rm gl.hl} A := \sup\{\hl(X^{\bullet}) \; | \; X^{\bullet} \in D^b(A)
\mbox{ is indecomposable}\}.$$ }\end{definition}

Obviously, the dimension of an $A$-module $M$ is equal to the
cohomological length of the stalk complex $M$. Note that there is a
full embedding of $\mod A$ into $D^b(A)$ which sends a module to the
corresponding stalk complex. If $\mbox{gl.hl} A < \infty$ then $A$
is representation-finite due to the truth of Brauer-Thrall
conjecture I.

\begin{definition}{\rm The {\it cohomological width} of a complex $X^{\bullet} \in D^b(A)$
is
$$\hw(X^{\bullet}) := \max\{j-i+1 \; | \; H^i(X^{\bullet}) \neq 0 \neq H^j(X^{\bullet})\},$$
and the {\it global cohomological width} of $A$ is
$$\mbox{\rm gl.hw} A := \sup\{\hw(X^{\bullet}) \; | \; X^{\bullet} \in D^b(A) \mbox{ is
indecomposable}\}.$$ }\end{definition}

Clearly, the cohomological width of a stalk complex is 1. If $A$ is
a hereditary algebra then every indecomposable complex $X^{\bullet}
\in D^b(A)$ is isomorphic to a stalk complex by \cite[I.5.2
Lemma]{Hap88}. Thus $\mbox{\rm gl.hw} A = 1$.

\begin{definition}{\rm The {\it cohomological range} of a complex $X^{\bullet} \in D^b(A)$
is
$$\hr(X^{\bullet}) := \hl(X^{\bullet}) \cdot \hw(X^{\bullet}),$$
and the {\it global cohomological range} of $A$ is
$$\mbox{\rm gl.hr} A := \sup\{\hr(X^{\bullet}) \; | \; X^{\bullet} \in D^b(A)
\mbox{ is indecomposable}\}.$$ }\end{definition}

The cohomological range of a complex will play a role similar to the
dimension of a module. It is invariant under shifts and
isomorphisms, since both cohomological length and cohomological
width are.

\medskip

Next we observe the behaviors of these invariants under derived
equivalences. For this, we need do some preparations.

Let $\mathcal{T}$ be a triangulated $k$-category with $[1]$ the
shift functor. For $T \in \mathcal{T}$, we define $\langle
T\rangle_n$ inductively by
$$\langle T \rangle_0 := \{X \in \mathcal{T} \; | \; X \mbox{ is a direct summand of }
T[i] \mbox{ for some } i \in \mathbb{Z}\},$$ and
$$\langle T\rangle_n := \left\{X \in \mathcal{T} \; \left| \;
\begin{array}{l} Y'\rightarrow X \oplus Y \rightarrow Y'' \rightarrow Y'[1] \mbox{ is
a triangle in } \mathcal{T} \\ \mbox{ with } Y', Y'' \in \langle T
\rangle_{n-1} \mbox{ and } Y \in \mathcal{T}
\end{array}\right.\right\}.$$ Clearly, $\langle T \rangle_{n-1} \subseteq
\langle T \rangle_{n}$ and $\langle T \rangle := \bigcup_{n \geq 0}
\langle T \rangle_{n}$ is the smallest thick subcategory of
$\mathcal{T}$ containing $T$. For $X \in \langle T \rangle$, the
{\it distance} of $X$ from $T$ is $$d(X, T) := \min\{n \in
\mathbb{N} \; | \; X \in \langle T \rangle_{n}\}.$$

\begin{lemma} \label{lemma-GK} {\rm (See Geiss and Krause \cite[Lemma
4.1]{GK02})} Let $\mathcal{T}$ be a triangulated $k$-category, $T
\in \mathcal{T}$ and $X \in \langle T \rangle$. Then for all $Y \in
\mathcal{T}$, $$\dim \; \Hom_{\mathcal{T}}(X, Y) \leq 2^{d(X,
T)}\sup\limits_{i \in \mathbb{Z}} \; \dim \;
\Hom_{\mathcal{T}}(T[i], Y).$$
\end{lemma}

\begin{proposition}
\label{prop-equi-h} Let $A$ and $B$ be two algebras,
$_AT^{\bullet}_B$ a two-sided tilting complex, and $F = -
\otimes^L_A T^{\bullet}_B : D^b(A) \rightarrow D^b(B)$ a derived
equivalence. Then there are $N_1, N_2, N_3 \in \mathbb{N}$ such that
for all $X^{\bullet} \in D^b(A)$,

{\rm (1)} $\hw(F(X^{\bullet})) \leq \hw(X^{\bullet}) + N_1$,

{\rm (2)} $\hl(F(X^{\bullet})) \leq N_2 \cdot \hl(X^{\bullet})$,

{\rm (3)} $\hr(F(X^{\bullet})) \leq N_3 \cdot \hr(X^{\bullet})$.
\end{proposition}

\begin{proof} (1) Recall that the {\it width} of a complex $Y^{\bullet} \in C^b(A)$
is $$\mbox{w}(Y^{\bullet}) := \max\{j-i+1 \; | \; Y^j \neq 0 \neq
Y^i\}.$$  For any $X^{\bullet} \in D^b(A)$, there exists a complex
$\tilde{X}^{\bullet} \in D^b(A)$ which can be obtained from
$X^{\bullet}$ by good truncations, such that
$\hw(\tilde{X}^{\bullet}) = \mbox{w}(\tilde{X}^{\bullet})$ and
$\tilde{X}^{\bullet} \cong X^{\bullet}$ in $D^b(A)$. Since
$_AT^{\bullet}_B$ is a two-sided tilting complex, there is a perfect
complex $_A\tilde{T}^\bullet \in C^b(\proj A^{\op})$ such that
$_AT^{\bullet} \cong \,_A\tilde{T}^\bullet$ in $D^b(A^\op)$. Thus
$F(\tilde{X}^{\bullet}) = \tilde{X}^{\bullet} \otimes^L_A
T^{\bullet} \cong \tilde{X}^{\bullet} \otimes_A \tilde{T}^{\bullet}$
in $D^b(k)$. Hence $\hw(F(X^{\bullet})) =
\hw(F(\tilde{X}^{\bullet})) = \hw(\tilde{X}^{\bullet} \otimes_A
\tilde{T}^{\bullet}) \leq \mbox{w}(\tilde{X}^{\bullet} \otimes_A
\tilde{T}^{\bullet}) \leq \mbox{w}(\tilde{X}^{\bullet}_A) +
\mbox{w}(_A\tilde{T}^{\bullet}) -1 = \hw(\tilde{X}^{\bullet}_A) +
\mbox{w}(_A\tilde{T}^{\bullet}) -1 = \hw(X^{\bullet}_A) +
\mbox{w}(_A\tilde{T}^{\bullet}) -1.$ So $N_1 :=
\mbox{w}(_A\tilde{T}^{\bullet})-1$ is as required.

\medskip

(2) Since $F$ is a derived equivalence, we have $B\in K^b(\proj
B)=\langle F(A) \rangle$. By Lemma~\ref{lemma-GK}, we get
$$\begin{array}{rcl} \dim \; H^i(F(X^{\bullet})) &
= & \dim \; \Hom_{D^b(B)}(B, F(X^{\bullet})[i]) \\
& \leq & 2^{d(B, F(A))} \;
\sup\limits_{j \in \mathbb{Z}} \; \dim \; \Hom_{D^b(B)}(F(A), F(X^{\bullet})[j]) \\
& = & 2^{d(B, F(A))} \;
\sup\limits_{j \in \mathbb{Z}} \; \dim \; \Hom_{D^b(A)}(A, X^{\bullet}[j]) \\
&
= & 2^{d(B, F(A))} \; \sup\limits_{j \in \mathbb{Z}} \; \dim \; H^j(X^{\bullet}) \\
& = & 2^{d(B, F(A))} \; \hl(X^{\bullet}).
\end{array}$$ Thus $N_2 := 2^{d(B, F(A))}$ is as required.

\medskip

(3) It follows from (1) and (2) that $\hr(F(X^{\bullet})) =
\hl(F(X^{\bullet})) \cdot \hw(F(X^{\bullet})) \leq N_2 \cdot
\hl(X^{\bullet}) \cdot (\hw(X^{\bullet})+N_1) \leq N_2(N_1+1) \cdot
\hr(X^{\bullet})$. Thus $N_3 := N_2(N_1+1)$ is as required.
\end{proof}

\begin{corollary}
\label{coro-equi-h} Let two algebras $A$ and $B$ be derived
equivalent. Then $\mbox{\rm gl.hw} A < \infty$ {\rm (}resp.
$\mbox{\rm gl.hl} A < \infty$, $\mbox{\rm gl.hr} A < \infty${\rm )}
if and only if $\mbox{\rm gl.hw} B < \infty$ {\rm (}resp. $\mbox{\rm
gl.hl} B < \infty$, $\mbox{\rm gl.hr} B < \infty${\rm )}.
\end{corollary}

\begin{proof} Since $A$ and $B$ are derived equivalent,
there is a two-sided tilting complex $_AT^{\bullet}_B$ such that $-
\otimes^L_A T^{\bullet}_B : D^b(A) \rightarrow D^b(B)$ is a derived
equivalence \cite{Ri91}. So the corollary follows from
Proposition~\ref{prop-equi-h}. \end{proof}

\subsection{\normalsize Strong global dimension}

Strong global dimension was introduced by Skowro\'{n}ski in
\cite{Sko87}. Happel and Zacharia characterized piecewise hereditary
algebras as the algebras of finite strong global dimension
\cite{HZ08}. Here, we adopt the definition of strong global
dimension in \cite{HZ08}, which is slightly different from that in
\cite{Sko87}.

\medskip

Recall that a complex $X^{\bullet}=(X^i, d^i) \in C(A)$ is said to
be {\it minimal} if $\Im d^i \subseteq \rad X^{i+1}$ for all $i \in
\mathbb{Z}$. For any complex $P^{\bullet}=(P^i, d^i) \in K^b(\proj
A)$, there is a minimal complex $\bar{P}^{\bullet}=(\bar{P}^i,
\bar{d}^i) \in K^b(\proj A)$, which is unique up to isomorphism in
$C^b(A)$, such that $P^{\bullet} \cong \bar{P}^{\bullet}$ in
$K^b(\proj A)$. The {\it length} of $P^{\bullet}$ is
$$l(P^{\bullet}) := \max\{j-i \; | \; \bar{P}^i \neq 0 \neq
\bar{P}^j \}.$$ The {\it strong global dimension} of $A$ is
$$\mbox{\rm s.gl.dim}A := \sup\{l(P^{\bullet}) \; | \; P^{\bullet} \in
K^b(\proj A) \mbox{ is indecomposable}\}.$$

Obviously, for a module of finite projective dimension, the length
of its minimal projective resolution equals to its projective
dimension. Furthermore, if $\mbox{\rm gl.dim} A < \infty$ then
$\mbox{\rm s.gl.dim} A \geq \mbox{\rm gl.dim} A$.

\medskip

The following result sets up the connection between the
indecomposable objects in $K^b(\proj A)$ and those in $K^{-,b}(\proj
A)$.

\begin{proposition} \label{prop-indec}
Let $P^{\bullet} \in K^{-,b}(\proj A)$ be a minimal complex and $n
:= \min\{i \in \mathbb{Z} \; | \; H^i(P^{\bullet}) \neq 0\}$. Then
$P^{\bullet}$ is indecomposable if and only if so is the brutal
truncation $\sigma_{\geq j}(P^{\bullet}) \in K^b(\proj A)$ for some
(resp. all) $j<n$.
\end{proposition}

\begin{proof}
Since $K^{-,b}(\proj A) \simeq D^b(A)$ is a Krull-Schmidt category,
a complex $X^{\bullet}\in K^{-,b}(\proj A)$ is indecomposable if and
only if its endomorphism algebra $\End_{K(A)}(X^{\bullet})$ is
local, if and only if $\End_{K(A)}(X^{\bullet}) /\rad
\End_{K(A)}(X^{\bullet}) \cong k$. Hence, it suffices to show
$$\End_{K(A)}(P^{\bullet}) / \rad \End_{K(A)}(P^{\bullet})
\cong \End_{K(A)}(\sigma_{\geq j}(P^{\bullet})) / \rad
\End_{K(A)}(\sigma_{\geq j}(P^{\bullet})).$$

Since $P^{\bullet}$ is minimal, all null homotopies in
$\End_{C(A)}(\sigma_{\geq j}(P^{\bullet}))$ form a nilpotent ideal
of $\End_{C(A)}(\sigma_{\geq j}(P^{\bullet}))$, thus are in $\rad
\End_{C(A)}(\sigma_{\geq j}(P^{\bullet}))$. Hence we have
$$\End_{K(A)}(\sigma_{\geq j}(P^{\bullet})) / \rad
\End_{K(A)}(\sigma_{\geq j}(P^{\bullet})) \cong
\End_{C(A)}(\sigma_{\geq j}(P^{\bullet}))/\rad
\End_{C(A)}(\sigma_{\geq j}(P^{\bullet})).$$ Now it is enough to
show
$$\End_{K(A)}(P^{\bullet}) / \rad \End_{K(A)}(P^{\bullet})
\cong \End_{C(A)}(\sigma_{\geq j}(P^{\bullet})) / \rad
\End_{C(A)}(\sigma_{\geq j}(P^{\bullet})).$$

Consider the composition of homomorphisms of algebras
$$\End_{C(A)}(P^{\bullet})\stackrel{\phi}{\rightarrow}
\End_{C(A)}(\sigma_{\geq j}(P^{\bullet}))
\stackrel{\psi}{\rightarrow} \End_{C(A)}(\sigma_{\geq
j}(P^{\bullet}))/\rad \End_{C(A)}(\sigma_{\geq j}(P^{\bullet})),$$
where $\phi$ is the natural restriction and $\psi$ is the canonical
epimorphism. Since $\sigma_{\leq j-1}(P^{\bullet})$ is a minimal
projective resolution of $\Ker d^j$, every cochain map in \linebreak
$\End_{C(A)}(\sigma_{\geq j}(P^{\bullet}))$ can be lifted to a
cochain map in $\End_{C(A)}(P^{\bullet})$, i.e., $\phi$ is
surjective. Thus the composition $\varphi := \psi\phi$ is
surjective. Since $P^{\bullet}$ is a minimal complex, all null
homotopies in $\End_{C(A)}(P^{\bullet})$ form a nilpotent ideal of
$\End_{C(A)}(P^{\bullet})$, thus are in
$\rad\End_{C(A)}(P^{\bullet})$. Furthermore, $\phi$ maps all null
homotopies in $\End_{C(A)}(P^{\bullet})$ into $\rad
\End_{C(A)}(\sigma_{\geq j}(P^{\bullet}))$. Hence $\varphi$ induces
a surjective homomorphism of algebras
$$\bar{\varphi}: \End_{K(A)}(P^{\bullet}) \twoheadrightarrow
\End_{C(A)}(\sigma_{\geq j}(P^{\bullet}))/\rad
\End_{C(A)}(\sigma_{\geq j}(P^{\bullet})).$$

Now it is sufficient to show that $\Ker \bar{\varphi} = \rad
\End_{K(A)}(P^{\bullet})$. Clearly, $\Ker \bar{\varphi} \supseteq
\rad \End_{K(A)}(P^{\bullet})$. Conversely, for any
$\bar{f^{\bullet}} \in \Ker \bar{\varphi}$ with $f^{\bullet} \in
\linebreak \End_{C(A)}(P^{\bullet})$, we have
$\psi(\phi(f^{\bullet})) = \bar{\varphi}(\bar{f^{\bullet}}) = 0$.
Thus $\phi(f^{\bullet})$ is nilpotent, i.e., there exists $t \in
\mathbb{N}$ such that $(f^i)^t=0$ for all $i \geq j$. Since
$\sigma_{\leq j-1}(P^{\bullet})$ is a minimal projective resolution
of $\Ker d^j$, the restriction $\sigma_{\leq j-1}(f^{\bullet}) \in
\End_{C(A)}(\sigma_{\leq j-1}(P^{\bullet}))$ of $f^{\bullet}$ is a
lift of the restriction of $f^j$ on $\Ker d^j$. Thus $(\sigma_{\leq
j-1}(f^{\bullet}))^t$ is a lift of the restriction of $(f^j)^t$ on
$\Ker d^j$, i.e., a lift of zero morphism. Hence $(\sigma_{\leq
j-1}(f^{\bullet}))^t$ is a null homotopy. Therefore,
$\bar{f^{\bullet}}^t=0$, i.e., $\bar{f^{\bullet}}$ is nilpotent, in
$\End_{K(A)}(P^{\bullet})$. So $\Ker\,\bar{\varphi}$ is a nilpotent
ideal of $\End_{K(A)}(P^{\bullet})$. Consequently,
$\Ker\,\bar{\varphi}\subseteq \rad \End_{K(A)}(P^{\bullet})$.
\end{proof}

\begin{corollary} \label{gl-sgl} Let $A$ be an algebra. Then
$\mbox{\rm gl.dim} A \leq \mbox{\rm s.gl.dim} A$. \end{corollary}

\begin{proof} We have known $\mbox{\rm gl.dim} A \leq
\mbox{\rm s.gl.dim} A$ if $\mbox{\rm gl.dim} A < \infty$. If
$\mbox{\rm gl.dim} A = \infty$ then there is a simple $A$-module $S$
of infinite projective dimension. Thus $S$ admits an infinite
minimal projective resolution. By Proposition~\ref{prop-indec},
there are indecomposable objects in $K^b(\proj A)$ of arbitrarily
large length, which implies $\mbox{\rm s.gl.dim} A = \infty$.
\end{proof}

\begin{remark}{\rm It is possible $\mbox{\rm gl.dim} A < \mbox{\rm s.gl.dim} A$.
Indeed, since piecewise hereditary algebras are factors of
finite-dimensional hereditary algebras \cite[Theorem 1.1]{HRS96},
all algebras of finite global dimension and with oriented cycles in
their quivers are not piecewise hereditary, thus of infinite strong
global dimension by \cite[Theorem 3.2]{HZ08}. }\end{remark}

As an additional corollary, we give a characterization of global
cohomological width on the level of bounded homotopy categories of
finite-dimensional projective modules.

\begin{corollary} \label{coro-refor-glhw} Let $A$ be an algebra. Then
$${\rm gl.hw} A=\sup\{\hw(P^\bullet)\; |\; P^\bullet
\; \mbox{\rm is (minimal) indecomposable in } K^b(\proj A) \}.$$
\end{corollary}

\begin{proof}
Clearly, the value of the right hand side of the equation is
invariant no matter whether we assume that the indecomposable
complex $P^\bullet \in K^b(\proj A)$ is minimal or not. Since
$K^b(\proj A) \subseteq D^b(A)$, the right hand side is not larger
than ${\rm gl.hw} A$. Conversely, by Proposition \ref{prop-indec},
any minimal indecomposable complex $P^\bullet \in K^{-,b}(\proj A)
\simeq D^b(A)$ has the property that $\hw(\sigma_{\geq
j}(P^\bullet))\geq \hw(P^\bullet)$ and $\sigma_{\geq j}(P^\bullet)
\in K^b(\proj A)$ is indecomposable for $j \ll 0$. Thus the right
hand side is not smaller than ${\rm gl.hw} A$.
\end{proof}

The following result implies that the global cohomological width can
provide an alternative definition of strong global dimension on the
level of bounded derived category.

\begin{proposition} \label{prop-eva-hw}
Let $A$ be an algebra. Then $\mbox{\rm gl.hw} A = \mbox{\rm
s.gl.dim} A.$
\end{proposition}

\begin{proof} First we show $\mbox{\rm gl.hw} A \geq \mbox{\rm s.gl.dim} A$.
It suffices to prove that for any minimal indecomposable complex
$P^{\bullet} \in K^b(\proj A)$ with $l(P^{\bullet})=n$, there is an
indecomposable complex $Q^\bullet \in  K^b(\proj A)$ such that
$\hw(Q^\bullet) \geq n$. Without loss of generality, we assume
$$P^{\bullet} = 0 \longrightarrow P^{-n} \stackrel{d^{-n}}{\longrightarrow}
P^{-n+1} \stackrel{d^{-n+1}}{\longrightarrow} \cdots
\stackrel{d^{-2}}{\longrightarrow} P^{-1}
\stackrel{d^{-1}}{\longrightarrow} P^0 \longrightarrow 0.$$ Since
$P^{\bullet}$ is minimal, we have $H^0(P^\bullet)\neq 0$. If
$H^{-n}(P^\bullet) \neq 0$ or $H^{-n+1}(P^\bullet) \neq 0$ then
$\hw(P^\bullet) \geq n$. Thus $Q^\bullet := P^\bullet$ is as
required. If $H^{-n}(P^\bullet) = 0 = H^{-n+1}(P^\bullet)$ then
$Q^\bullet := \sigma_{\geq -n+1}(P^\bullet)$ is as required, since
it is indecomposable by Proposition \ref{prop-indec} and
$\hw(Q^\bullet)=n$ .

Next we show $\mbox{\rm s.gl.dim} A \geq \mbox{\rm gl.hw} A$. By
Corollary \ref{coro-refor-glhw}, it is enough to show that for any
minimal indecomposable complex $P^{\bullet} \in K^b(\proj A)$, there
is a minimal indecomposable complex $Q^\bullet \in K^b(\proj A)$
such that $l(Q^\bullet) \geq \hw(P^\bullet)$. Without loss of
generality, we still assume that $P^\bullet$ is of the above form.
If $H^{-n}(P^{\bullet})=0$ then $l(P^\bullet) \geq \hw(P^\bullet)$.
Thus $Q^\bullet := P^\bullet$ is as required. If $H^{-n}(P^\bullet)
\cong \Ker d^{-n} \neq 0$, we take a minimal projective resolution
of $\Ker d^{-n}$, say
$$P'^{\bullet} = \cdots \longrightarrow P^{-n-2}
\stackrel{d^{-n-2}}{\longrightarrow} P^{-n-1} \longrightarrow 0.$$
Gluing $P'^{\bullet}$ and $P^{\bullet}$ together, we get a minimal
complex
$$P''^{\bullet} = \cdots \longrightarrow P^{-n-2}
\stackrel{d^{-n-2}}{\longrightarrow} P^{-n-1}
\stackrel{d^{-n-1}}{\longrightarrow} P^{-n}
\stackrel{d^{-n}}{\longrightarrow} \cdots
\stackrel{d^{-1}}{\longrightarrow} P^0 \longrightarrow 0,$$ where
$d^{-n-1}$ is the composition $P^{-n-1} \twoheadrightarrow \Ker
d^{-n} \hookrightarrow P^{-n}$. Since $P^{\bullet} = \sigma_{\geq
-n}(P''^{\bullet})$ is indecomposable and $H^i(P''^{\bullet}) =0$
for all $i \leq -n$, by Proposition~\ref{prop-indec},
$P''^{\bullet}$ is indecomposable. Also by
Proposition~\ref{prop-indec}, we have $Q^\bullet := \sigma_{\geq
-n-1}(P''^\bullet)$ is indecomposable with $l(Q^\bullet) = n+1 =
\hw(P^\bullet)$.
\end{proof}

\medskip

Recall that an algebra $A$ is said to be {\it piecewise hereditary}
if there is a triangle equivalence $D^b(A) \simeq D^b(\mathcal{H})$
for some hereditary abelian $k$-category $\mathcal{H}$ (Ref.
\cite{HRS96}). In this case, $\mathcal{H}$ must have a tilting
object \cite{HR98}. Thus there are exactly two classes of piecewise
hereditary algebras whose derived categories are triangle equivalent
to either $D^b(kQ)$ for some finite connected quiver $Q$ without
oriented cycles, or $D^b(\mbox{coh} \mathbb{X})$ for some weighted
projective line $\mathbb{X}$ (Ref. \cite{Hap01}).

As a corollary of Proposition~\ref{prop-eva-hw}, piecewise
hereditary algebras can be characterized as the algebras of finite
global cohomological width.

\begin{corollary} \label{coro-char-phere}
An algebra $A$ is piecewise hereditary if and only if $\mbox{\rm
gl.hw} A < \infty$.
\end{corollary}

\begin{proof} It follows from \cite[Theorem 3.2]{HZ08}
and Proposition~\ref{prop-eva-hw}. \end{proof}

\subsection{\normalsize The first Brauer-Thrall type theorem}

\begin{definition}{\rm An algebra $A$ is said to be {\it derived bounded} if $\mbox{\rm gl.hr}
A<\infty$, i.e., the cohomological ranges of all indecomposable
objects in $D^b(A)$ have a common upper bound. }\end{definition}

Recall that an algebra $A$ is said to be {\it derived finite} if up
to shift and isomorphism there are only finitely many indecomposable
objects in $D^b(A)$ (Ref. \cite{BM03}). Now we can prove Theorem I,
another proof of which will be given at the end of this paper
(Remark~\ref{Remark-AnotherProofOfTheoremI}).

\begin{theorem}
\label{theorem-dbt1} Let $A$ be an algebra. Then the following
assertions are equivalent:

{\rm (1)} $A$ is derived bounded;

{\rm (2)} $A$ is derived finite;

{\rm (3)} $A$ is piecewise hereditary of Dynkin type.
\end{theorem}

\begin{proof} (1) $\Rightarrow$ (3): By assumption,
$\mbox{\rm gl.hr} A < \infty$. Thus $\mbox{\rm gl.hw} A < \infty$.
It follows from Corollary \ref{coro-char-phere} that $A$ is
piecewise hereditary. By \cite[Theorem 3.1]{Hap01}, $D^b(A) \simeq
D^b(kQ)$ for some finite connected quiver $Q$ without oriented
cycles, or $D^b(A) \simeq D^b(\mbox{coh}\mathbb{X})$ for some
weighted projective line $\mathbb{X}$. In the first case, by
Corollary~\ref{coro-equi-h} we have $\mbox{\rm gl.hr} \;kQ <
\infty$. Hence $Q$ is a Dynkin quiver. In the second case, by
\cite[Theorem 3.2]{GL87}, $D^b(A)$ is triangle equivalent to
$D^b(C)$ for a canonical algebra $C$. Since $C$ is
representation-infinite, $\mbox{\rm gl.hr} C = \infty$. By
Corollary~\ref{coro-equi-h}, we have $\mbox{\rm gl.hr} A = \infty$,
which is a contradiction.

(3) $\Rightarrow$ (2): This is well-known \cite{Hap88}.

(2) $\Rightarrow$ (1): Trivial.
\end{proof}

\section{\large The second Brauer-Thrall type theorem}

\subsection{\normalsize Strongly derived unbounded algebras}

Recall that the {\it cohomology dimension vector} of a complex
$X^{\bullet} \in D^b(A)$ is $\mathbf{d}(X^{\bullet}) := (\dim
H^n(X^{\bullet}))_{n \in \mathbb{Z}} \in \mathbb{N}^{(\mathbb{Z})}$.
An algebra $A$ is said to be {\it derived discrete} if for any
$\mathbf{d} \in \mathbb{N}^{(\mathbb{Z})}$, up to isomorphism, there
are only finitely many indecomposable objects in $D^b(A)$ of
cohomology dimension vector $\mathbf{d}$ (Ref. \cite{Vo01}). It is
easy to see that an algebra $A$ is derived discrete if and only if
for any $r \in \mathbb{N}$, up to shift and isomorphism, there are
only finitely many indecomposable objects in $D^b(A)$ of
cohomological range $r$.

\begin{definition}{\rm An algebra $A$ is said to be {\it strongly derived unbounded} if
there is an (strictly) increasing sequence $\{r_i \; | \; i \in
\mathbb{N}\} \subseteq \mathbb{N}$ such that for each $r_i$, up to
shift and isomorphism, there are infinitely many indecomposable
objects in $D^b(A)$ of cohomological range $r_i$.} \end{definition}

Note that all representation-infinite algebras are strongly
unbounded due to the truth of Brauer-Thrall conjecture II, thus
strongly derived unbounded. Moreover, it is impossible that an
algebra is both derived discrete and strongly derived unbounded.

Now we show that derived equivalences preserve strongly derived
unboundedness.

\begin{proposition} \label{prop-der-sdu} Let two algebras $A$ and $B$
be derived equivalent. Then $A$ is strongly derived unbounded if and
only if so is $B$.
\end{proposition}

\begin{proof} Let $_AT^{\bullet}_B$ be a two-sided tilting complex
such that $F = - \otimes^L_A T^{\bullet}_B : D^b(A) \rightarrow
D^b(B)$ is a derived equivalence. Assume that $A$ is strongly
derived unbounded. Then there exist an increasing sequence $\{r_i \;
| \; i \in \mathbb{N}\} \subseteq \mathbb{N}$ and infinitely many
indecomposable objects $\{X_{ij}^{\bullet} \in D^b(A) \; | \; i, j
\in \mathbb{N}\}$ which are pairwise different up to shift and
isomorphism such that $\hr(X_{ij}^{\bullet}) = r_i$ for all $j \in
\mathbb{N}$. It follows from Proposition \ref{prop-equi-h} (3) that
there exist two positive integers $N$ and $N'$, such that for any
$X_{ij}^{\bullet}$,
$$\frac{1}{N'} \cdot \hr(X_{ij}^{\bullet}) \leq \hr(F(X_{ij}^{\bullet})) \leq N \cdot \hr(X_{ij}^{\bullet}).$$
In order to show that $B$ is strongly derived unbounded, we shall
find inductively an increasing sequence $\{r'_i \; | \; i \in \mathbb{N}\} \subseteq \mathbb{N}$ and
infinitely many indecomposable objects $\{Y_{ij}^{\bullet}\in D^b(B)
\; | \; i, j \in \mathbb{N}\}$ which are pairwise different up to
shift and isomorphism such that $\hr(Y_{ij}^{\bullet}) = r'_i$ for
all $j \in \mathbb{N}$. For $i=1$, we have $0<
\hr(F(X_{1j}^{\bullet})) \leq N \cdot \hr(X_{1j}^{\bullet}) = N
\cdot r_1$. Since $F(X_{1j}^{\bullet}), j \in \mathbb{N}$, are also
pairwise different indecomposable objects up to shift and
isomorphism, we can choose $0< r'_1 \leq Nr_1$ and infinitely many
indecomposable objects $\{Y^{\bullet}_{1j} \; | \; j \in
\mathbb{N}\} \subseteq \{F(X_{1j}^{\bullet}) \; | \; j \in
\mathbb{N}\} $ which are pairwise different up to shift and
isomorphism such that $\hr(Y^{\bullet}_{1j}) = r'_1$ for all $j \in
\mathbb{N}$. Assume that we have found $r'_i$. We choose some $r_l$
with $r_l > N' \cdot r'_i$. Since
$$r'_i < \frac{1}{N'} \cdot r_l = \frac{1}{N'} \cdot \hr(X_{lj}^{\bullet})
\leq \hr(F(X_{lj}^{\bullet})) \leq N \cdot \hr(X_{lj}^{\bullet}) = N
\cdot r_l,$$ we can choose $r'_i < r'_{i+1} \leq N \cdot r_l$ and
infinitely many indecomposable objects $\{Y^{\bullet}_{i+1,j} \; |
\; j \in \mathbb{N}\} \subseteq \{F(X_{lj}^{\bullet}) \; | \; j \in
\mathbb{N}\}$ which are pairwise different up to shift and
isomorphism such that $\hr(Y^{\bullet}_{i+1,j}) = r'_{i+1}$ for all
$j \in \mathbb{N}$.
\end{proof}

\subsection{\normalsize Cleaving functors}

Cleaving functors were introduced in \cite{BGRS85} as a tool for
proving that certain algebras are representation-infinite. In this
part, we will observe the relations between cleaving functors and
strongly derived unboundedness of algebras.

In order to use cleaving functors, one needs to view basic
finite-dimensional algebras or bound quiver algebras as finite
spectroids. Recall that a {\it locally bounded spectroid}
\cite{GR97} (=locally bounded category \cite{BoG82}) is a small
$k$-linear category $A$ satisfying:

(1) different objects in $A$ are not isomorphic;

(2) the endomorphism algebra $A(a,a)$ is local for all $a \in A$;

(3) $\sum_{x \in A} \dim \; A(a,x) < \infty$ and $\sum_{x \in A}
\dim \; A(x, a) < \infty$ for all $a \in A$.

\noindent A {\it finite spectroid} is a locally bounded spectroid
with finitely many objects. Let $A$ be a finite spectroid. A {\it
right $A$-module} $M$ is just a covariant $k$-functor from $A$ to
the category of $k$-vector spaces. The {\it dimension} of $M$ is
$\dim M := \sum_{a \in A}\dim M(a)$. Denote by $\mod A$ the category
of finite-dimensional right $A$-modules. The indecomposable
projective $A$-modules are $P_a=A(a, -)$ and indecomposable
injective $A$-modules are $I_a=DA(-, a)$ for all $a \in A$, where
$D=\Hom_k(-, k)$. A bound quiver algebra $kQ/I$ with $Q$ a finite
quiver and $I$ an admissible ideal can be viewed as a finite
spectroid $A$ by taking the vertices in $Q_0$ as objects and the
$k$-linear combinations of paths in $kQ/I$ as morphisms. Conversely,
a finite spectroid $A$ admits a presentation $kQ/I
\stackrel{\sim}{\rightarrow} A$ for a finite quiver $Q$ and an
admissible ideal $I$ (Ref. \cite[Chapter 8]{GR97}). In these cases,
$kQ/I$ and $A$ have equivalent (finite-dimensional) module
categories. {\it Throughout this section, we do not differentiate
the terminologies ``(basic finite-dimensional) algebra'', ``bound
quiver algebra'' and ``finite spectroid''. In particular, all the
concepts and notations concerning module category defined for a
bound quiver algebra make sense for a finite spectroid}.

To a $k$-functor $F: B \rightarrow A$ between finite spectroids, we
associates a {\it restriction functor} $F_{\ast}: \mod A \rightarrow
\mod B$, which is given by $F_{\ast}(M) = M \circ F$ and exact. The
restriction functor $F_{\ast}$ admits a left adjoint functor
$F^{\ast}$, called the {\it extension functor}, which sends a
projective $B$-module $B(b, -)$ to a projective $A$-module $A(Fb,
-)$. If $\gl B < \infty$ then $F_{\ast}$ extends naturally to a
derived functor $F_{\ast}: D^b(A) \rightarrow D^b(B)$, which has a
left adjoint $\mathbf{L}F^{\ast}: D^b(B)\rightarrow D^b(A)$. Note
that $\mathbf{L}F^{\ast}$ is the left derived functor associated
with $F^{\ast}$ and maps $K^b(\proj B)$ into $K^b(\proj A)$. We
refer to \cite{W94} for the definition of derived functors.

A $k$-functor $F: B \rightarrow A$ between finite spectroids with
$\gl B < \infty$ is called a {\it cleaving functor}
\cite{BGRS85,Vo01} if it satisfies the following equivalent
conditions:

(1) The linear map $B(b,b') \rightarrow A(Fb,Fb')$ associated with
$F$ admits a natural retraction for all $b,b' \in B$;

(2) The adjunction morphism $\phi_M: M \rightarrow (F_{\ast} \circ
F^{\ast})(M)$ admits a natural retraction for all $M \in \mod B$;

(3) The adjunction morphism $\Phi_{X^{\bullet}}: X^{\bullet}
\rightarrow (F_{\ast} \circ \mathbf{L}F^{\ast})(X^{\bullet})$ admits
a natural retraction for all $X^{\bullet} \in D^b(B)$.

\begin{proposition}
\label{prop-cf-sdu} Let $F: B\rightarrow A$ be a cleaving functor
between finite spectroids with $\mbox{\rm gl.dim} B < \infty$. Then
the following assertions hold:

{\rm (1)} If $B$ is strongly derived unbounded then so is $A$.

{\rm (2)} If ${\rm gl.hl} A < \infty$ {\rm (}resp. ${\rm gl.hw} A <
\infty$, ${\rm gl.hr} A < \infty${\rm )}, then ${\rm gl.hl} B <
\infty$ {\rm (}resp. ${\rm gl.hw} B < \infty$, ${\rm gl.hr} B <
\infty${\rm )}.
\end{proposition}

\begin{proof} (1) Assume that there exist an increasing sequence
$\{r_i \; | \; i \in \mathbb{N}\} \subseteq \mathbb{N}$ and
indecomposable objects $\{X_{ij}^{\bullet} \in D^b(B) \; | \; i, j
\in \mathbb{N}\}$ which are pairwise different up to shift and
isomorphism such that $\hr(X_{ij}^{\bullet}) = r_i$ for all $j \in
\mathbb{N}$. Since $F$ is a cleaving functor, $X_{ij}^{\bullet}$ is
a direct summand of $(F_{\ast} \circ \mathbf{L}F^{\ast})
(X_{ij}^{\bullet})$. Thus for any $X_{ij}^{\bullet}$, we can choose
an indecomposable direct summand $Y_{ij}^{\bullet}$ of
$\mathbf{L}F^{\ast}(X_{ij}^{\bullet})$, such that $X_{ij}^{\bullet}$
is a direct summand of $F_{\ast}(Y_{ij}^{\bullet})$. Clearly, for
any $i \in \mathbb{N}$, the set $\{Y_{ij}^{\bullet} \; | \; j \in
\mathbb{N}\}$ contains infinitely many elements which are pairwise
different up to shift and isomorphism. To prove $A$ is strongly
derived unbounded, by the proof of Proposition \ref{prop-der-sdu},
it is enough to show that there exist $N', N \in \mathbb{N}$ such
that for any $X_{ij}^{\bullet}$, the inequalities $\frac{1}{N'}
\cdot \hr(X_{ij}^{\bullet}) \leq \hr(Y_{ij}^{\bullet}) \leq N \cdot
\hr(X_{ij}^{\bullet})$ hold.

For any $a \in A$, we have
$$\begin{array}{rcl} H^m(\mathbf{L}F^{\ast}(X_{ij}^{\bullet}))(a) &
\cong & \Hom_{D^b(A)}(\mathbf{L}F^{\ast}(X_{ij}^{\bullet}), I_a[m])
\\ & \cong & \Hom_{D^b(B)}(X_{ij}^{\bullet}, F_{\ast}(I_a)[m]) \\ &
\cong & H^m(\RHom_B(X_{ij}^{\bullet}, F_{\ast}(I_a))).
\end{array}$$
Since $\mbox{\rm gl.dim} B < \infty$, the $B$-module $F_{\ast}(I_a)$
admits a minimal injective resolution
$$0 \rightarrow F_{\ast}(I_a) \rightarrow E_a^0 \rightarrow E_a^1
\rightarrow \cdots \rightarrow E_a^{r_a} \rightarrow 0,$$ and there is a
bounded converging spectral sequence
$$\Ext_B^p(H^{-q}(X_{ij}^{\bullet}), F_{\ast}(I_a))
\Rightarrow H^{p+q}(\RHom_B(X_{ij}^{\bullet}, F_{\ast}(I_a))).$$
Thus $\hw(Y_{ij}^{\bullet}) \leq
\hw(\mathbf{L}F^{\ast}(X_{ij}^{\bullet})) \leq \hw(X_{ij}^{\bullet})
+ \mbox{\rm gl.dim}B$, and
$$\begin{array}{rl} \dim H^m(Y_{ij}^{\bullet}) & = \sum\limits_{a \in A} \dim
H^m(Y_{ij}^{\bullet})(a) \\
& \leq \sum\limits_{a \in A} \dim H^m(\mathbf{L}F^{\ast}(X_{ij}^{\bullet}))(a) \\
& = \sum\limits_{a \in A} \dim H^m(\RHom_B(X_{ij}^{\bullet}, F_{\ast}(I_a))) \\
& \leq \sum\limits_{a \in A} \sum\limits_{p+q=m} \dim \Ext_B^p(H^{-q}(X_{ij}^{\bullet}), F_{\ast}(I_a)) \\
& \leq \sum\limits_{a \in A}
\sum\limits_{p=0}^{r_a} \dim H^{p-m}(X_{ij}^{\bullet}) \cdot \dim E_a^p \\
& \leq \sum\limits_{a \in A} \hl(X_{ij}^{\bullet}) \cdot (r_a + 1)
\cdot
\max\limits_{0 \leq p \leq r_a} \{\dim E_a^p\} \\
& \leq n_0(A) \cdot \hl(X_{ij}^{\bullet}) \cdot (\gl B + 1) \cdot
\max\limits_{a \in A, \; 0 \leq p \leq r_a}\{\dim E_a^p\},
\end{array}$$
where $n_0(A)$ denotes the number of objects in $A$.

Set $N_0 = n_0(A) \cdot (\gl B + 1) \cdot \max\limits_{a \in A, \; 1
\leq p \leq r_a} \{\dim E_a^p\}$. Then $\hl(Y_{ij}^{\bullet}) \leq
N_0 \cdot \hl(X_{ij}^{\bullet})$ and
$$\begin{array}{rcl} \hr(Y_{ij}^{\bullet}) & = & \hw(Y_{ij}^{\bullet}) \cdot
\hl(Y_{ij}^{\bullet}) \\ & \leq & (\hw(X_{ij}^{\bullet}) + \mbox{\rm
gl.dim}B) \cdot N_0 \cdot \hl(X_{ij}^{\bullet}) \\
& \leq & N_0 \cdot (\mbox{\rm gl.dim }B+1) \cdot
\hr(X_{ij}^{\bullet}). \end{array}$$ So $N := N_0 \cdot (\mbox{\rm
gl.dim }B+1)$ is as required.

Assume the indecomposable projective $B$-module $Q_b=B(b, -)$ and
indecomposable projective $A$-module $P_a = A(a, -)$ for all $b \in
B$ and $a \in A$. Then
$$\begin{array}{rcl} \dim H^{m}(X_{ij}^{\bullet}) & \leq & \dim
H^m(F_{\ast}(Y_{ij}^{\bullet})) \\
& = & \sum\limits_{b \in B}\dim \Hom_{D^b(B)}(Q_b, F_{\ast}(Y_{ij}^{\bullet})[m])\\
& = & \sum\limits_{b \in B}\dim \Hom_{D^b(A)}(\mathbf{L}F^{\ast}(Q_b), Y_{ij}^{\bullet}[m])\\
& = & \sum\limits_{b \in B}\dim \Hom_{D^b(A)}(F^{\ast}(Q_b), Y_{ij}^{\bullet}[m])\\
& = & \sum\limits_{b \in B}\dim \Hom_{D^b(A)}(P_{F(b)}, Y_{ij}^{\bullet}[m])\\
& \leq & n_0(B) \cdot \dim \Hom_{D^b(A)}(A, Y_{ij}^{\bullet}[m])\\
& = & n_0(B) \cdot \dim H^m(Y_{ij}^{\bullet})
\end{array}$$ for all $m \in \mathbb{Z}$, where $n_0(B)$ denotes the number of objects in
$B$. Thus $\hl(X_{ij}^{\bullet}) \leq n_0(B) \cdot
\hl(Y_{ij}^{\bullet})$, $\hw(X_{ij}^{\bullet}) \leq
\hw(Y_{ij}^{\bullet})$, and $\hr(Y_{ij}^{\bullet}) \geq
\frac{1}{n_0(B)} \cdot \hr(X_{ij}^{\bullet})$. So $N' := n_0(B)$ is
as required.

(2) It can be read off from the proof of (1) that for any indecomposable object $X^{\bullet}\in D^b(B)$,
there exists an indecomposable object $Y^{\bullet}\in D^b(A)$ such that
$\dim H^{m}(X^{\bullet}) \leq n_0(B) \cdot \dim H^m(Y^{\bullet})$ for all $m\in \mathbb{Z}$. Then
the statement follows.
\end{proof}

\subsection{\normalsize Simply connected algebras}

To a tilting $A$-module $T_A$, one can associate a torsion pair
$(\mathcal{T}(T), \mathcal{F}(T))$ in $\mod A$, and a torsion pair
$(\mathcal{X}(T), \mathcal{Y}(T))$ in $\mod \End_A(T)$. The
Brenner-Butler theorem in classical tilting theory establishes the
equivalence between $\mathcal{F}(T)$ and $\mathcal{X}(T)$ under the
restriction of functor $F=\Ext^1_A(T_A, -)$, and the equivalence
between $\mathcal{T}(T)$ and $\mathcal{Y}(T)$ under the restriction
of functor $G=\Hom_A(T_A, -)$ (Ref. \cite[Theorem (2.1)]{HR82}). We
say a torsion pair $(\mathcal{T}, \mathcal{F})$ in $\mod A$ {\it
splits} if any indecomposable $M$ in $\mod A$ is either in
$\mathcal{T}$ or in $\mathcal{F}$. A tilting $A$-module $T$ is said
to be {\it separating} if the torsion pair $(\mathcal{T}(T),
\mathcal{F}(T))$ splits. A tilting $A$-module $T$ is said to be {\it
splitting} if the torsion pair $(\mathcal{X}(T), \mathcal{Y}(T))$
splits. Refer to \cite[Chapter VI]{ASS06}.

Recall that an algebra $A$ is said to be {\it triangular} if its
quiver $Q_A$ has no oriented cycles. A triangular algebra $A$ is
said to be {\it simply connected} if for any presentation $A \cong
kQ/I$, the fundamental group $\Pi_1(Q, I)$ of $(Q, I)$ is trivial
\cite{MP83}. Now we prove Theorem II for simply connected algebras.

\begin{lemma}
\label{lemma-dbt2-sc} A simply connected algebra $A$ is either
derived discrete or strongly derived unbounded. Moreover, a simply
connected algebra $A$ is derived discrete if and only if it is
piecewise hereditary of Dynkin type, if and only if ${\rm gl.hl}A <
\infty$.
\end{lemma}

\begin{proof} According to Corollary \ref{coro-equi-h} and Proposition
\ref{prop-der-sdu}, it is enough to show that a simply connected
algebra $A$ is tilting equivalent (thus derived equivalent) to
either a hereditary algebra or a representation-infinite algebra. If
$A$ is itself hereditary or representation-infinite then we have
nothing to do. If $A$ is representation-finite but not hereditary
then, by \cite[Theorem]{Ass87}, there exists a separating but not
splitting basic tilting $A$-module $T$. Put $A_1 = \End_A(T)$. Then
there are more indecomposable modules in $\mod A_1$ than in $\mod
A$, in particular $A$ and $A_1$ are not isomorphic as algebras.
Moreover, $A_1$ is still simply connected by \cite[Corollary]{AS94}
and thus triangular. Since $A_1$ is a tilted algebra of $A$, they
have the same number of simple modules \cite[Corollary (3.1)]{HR82}.
If $A_1$ is hereditary or representation-infinite then we have
nothing to do. If $A_1$ is representation-finite but not hereditary
then there exists a separating but not splitting basic tilting
$A_1$-module, and we can proceed as above repetitively. We claim
this process must stop in finite steps, and thus $A$ is tilting
equivalent to either a hereditary algebra or a
representation-infinite algebra. Indeed, for any $n \in \mathbb{N}$,
there are only finitely many (unnecessarily connected) basic
representation-finite triangular algebras having $n$ simple modules
up to isomorphism (compare with \cite[Chapter IV, Lemma
7.3]{Hap88}). We can prove this by induction on $n$. If $n=1$, then
there exists only one basic triangular algebra up to isomorphism.
Assume that it is true for $n-1$ and $B$ is a basic
representation-finite triangular algebra having $n$ simple modules.
Then $B$ is a one-point extension of a basic representation-finite
triangular algebra with $n-1$ simples, say $C$, by some $C$-module
$M = \oplus_{i=1}^r M_i$ with $M_i$ being indecomposable. Since $C$
is representation-finite, we have $r \leq 3$. Indeed, if $r \geq 4$
then $\dim \; e(\rad B / \rad^2 B)(1-e) = \dim \; M/\rad M \geq 4$,
where $e$ is the idempotent of $B$ corresponding to the extension
vertex. Thus in the quiver of $B$ there will be at least four arrows
starting from the extension vertex, which implies that $B$ is
representation-infinite. It is a contradiction. Therefore, the
number of the isomorphism classes of basic representation-finite
triangular algebras having $n$ simple modules is finite.
Furthermore, the tilting process above must stop in finite steps,
since representation-finite simply connectedness and the number of
simples are invariant under this process.
\end{proof}

\subsection{\normalsize The second Brauer-Thrall type theorem}

Bekkert and Merklen have classified the indecomposable objects in
the derived category of a gentle algebra \cite[Theorem 3]{BM03}. Now
we apply their results to prove Theorem II for gentle algebras.

\begin{lemma}
\label{lemma-dbt2-gentle} A gentle algebra $A$ is either derived
discrete or strongly derived unbounded. Moreover, $A$ is derived
discrete if and only if ${\rm gl.hl} A < \infty$.
\end{lemma}

\begin{proof} It follows from \cite[Theorem 4]{BM03} that a gentle
algebra $A$ is derived discrete if and only if $A$ does not contain
generalized bands.

If $A$ contains a generalized band $w$, then one can construct
indecomposable complexes $\{P^{\bullet}_{w,f} \; | \;
f=(x-\lambda)^d \in k[x], \; \lambda \in k^*, \; d>0 \}$ which are
pairwise different up to shift and isomorphism such that
$P^{\bullet}_{w,f}$ and $P^{\bullet}_{w,f'}$ are of the same
cohomological range (resp. cohomological length) if and only if
$\deg(f) = \deg(f')$ (Ref. \cite[Definition 3]{BM03}). Thus $A$ is
strongly derived unbounded and ${\rm gl.hl} A = \infty$.

If $A=kQ/I$ does not contain generalized bands we shall prove ${\rm
gl.hl} A <\infty$. By Bobi{\'n}ski, Geiss and Skowro{\'n}ski's
classification of derived discrete algebras \cite[Theorem A]{BGS04},
we know that $A$ is derived equivalent to a gentle algebra
$\Lambda(r,n,m)$ with $n \geq r \geq 1$ and $m \geq 0$, which is
given by the quiver {\tiny 
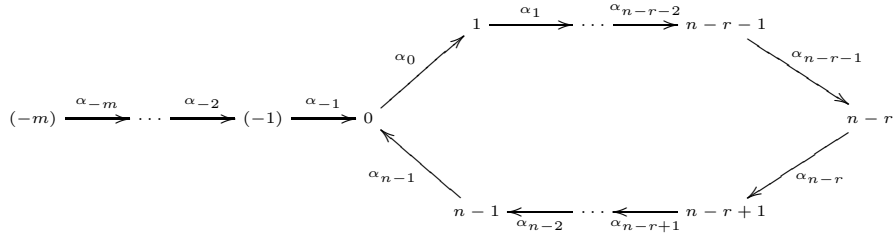
$$\xymatrix{ &&&& 1 \ar[r]^-{\alpha_{1}}
& \cdots \ar[r]^-{\alpha_{n-r-2}} &  n-r-1 \ar[dr]^-{\alpha_{n-r-1}} & \\
(-m) \ar[r]^-{\alpha_{-m}} & \cdots \ar[r]^-{\alpha_{-2}} & (-1)
\ar[r]^-{\alpha_{-1}} & 0 \ar[ur]^-{\alpha_0} &&&& n-r \ar[dl]^-{\alpha_{n-r}} \\
&&&& n-1 \ar[ul]^-{\alpha_{n-1}} & \cdots \ar[l]^-{\alpha_{n-2}} &
n-r+1 \ar[l]^-{\alpha_{n-r+1}} & }$$}

\noindent with the relations $\alpha_{n-1}\alpha_0,
\alpha_{n-2}\alpha_{n-1}, \cdots , \alpha_{n-r}\alpha_{n-r+1}$.
According to Corollary \ref{coro-equi-h}, it suffices to show that
${\rm gl.hl}\Lambda(r,n,m) \leq \dim \Lambda(r,n,m)<\infty$. Note
that any generalized string $w$ of $\Lambda(r,n,m)$ must be a
sub-generalized string of the following generalized strings or their
inverses:
$$(\alpha_i\cdots\alpha_{n-r})[(\alpha_{n-r+1}) \cdots
(\alpha_{n-1})(\alpha_0\cdots \alpha_{n-r})]^p \, (\alpha_{n-r+1})
\cdots (\alpha_{n-1})(\alpha_j\cdots \alpha_{-1})^{-1},$$ with
$-m\leq i\leq n-r$, $-m\leq j\leq -1$ and $p\geq0$. By Bekkert and
Merklen's construction of the indecomposable objects in the bounded
derived category of a gentle algebra \cite[Definition 2 and Theorem
3]{BM03}, every indecomposable projective direct summand of each
component of the indecomposable object $P^\bullet_w\in
K^b(\mbox{proj} \Lambda(r,n,m))$ is multiplicity-free, and hence
${\rm gl.hl}\Lambda(r,n,m) \leq \dim \Lambda(r,n,m)<\infty$.
\end{proof}

Let $A_n^m$ be the finite spectroid defined by the quiver
$$\xymatrix{
n \ar [r]^{\alpha_{n-1}} & n-1\ar [r]^{\alpha_{n-2}}& \cdots \ar
[r]^{\alpha_2}& 2 \ar [r]^{\alpha_1} &1},$$ and the admissible ideal
generated by all paths of length $m$.

\begin{lemma}
\label{lemma-sdu-a} The finite spectroid $A_{3m}^m$ with $m \geq 3$
is strongly derived unbounded and ${\rm gl.hl} A_{3m}^m = \infty$.
\end{lemma}

\begin{proof} Let $B = A_{3m}^m$, $w_1=\alpha_{3m-1}$, $w_{2}=\alpha_{3m-2}\cdots\alpha_{2m}$,
$w_{3}=\alpha_{2m-1}\cdots\alpha_{m+1}$,
$w_{4}=\alpha_{m}\cdots\alpha_{2}$, $w_5=\alpha_{1}$,
$w'_1=\alpha_{3m-1}\cdots\alpha_{2m+1}$, $w'_2=\alpha_{2m}$,
$w'_3=w_3$, $w'_{4}=\alpha_{m}$, and
$w'_5=\alpha_{m-1}\cdots\alpha_{1}$. Then we construct a family of
complexes $\{P^{\bullet}_{\lambda, d} \; | \; \lambda \in k, d \geq
1\}$ by
$$\begin{array}{ll} P^{\bullet}_{\lambda, d} :=0 & \!\!\! \rightarrow P_1^d
\stackrel{\delta^0}{\rightarrow} P_m^d \oplus P_2^d
\stackrel{\delta^1}{\rightarrow} P_{m+1}^d \oplus P_{m+1}^d
\stackrel{\delta^2}{\rightarrow} P_{2m}^d \oplus P_{2m}^d \\
& \!\!\! \stackrel{\delta^3}{\rightarrow} P_{2m+1}^d \oplus
P_{3m-1}^d \stackrel{\delta^4}{\rightarrow} P_{3m}^d \rightarrow 0
\end{array}$$ with the differential maps
$$\delta^0 := \left(\begin{array}{cc} P(w'_5) \mathbf{I}_d \\
P(w_5) \mathbf{J}_{\lambda,d} \end{array}\right), \;\;
\delta^i := \left(\begin{array}{cc} P(w'_{5-i}) \mathbf{I}_d & 0 \\
0 & P(w_{5-i}) \mathbf{I}_d \end{array}\right), \;\; \mbox{ for
}i=1, 2, 3,$$ and $\delta^4 := (P(w'_1) \mathbf{I}_d, P(w_1)
\mathbf{I}_d)$. Here $\mathbf{J}_{\lambda, d}$ denotes the upper
triangular $d \times d$ Jordan block with eigenvalue $\lambda \in
k$, and the map $P(u)$ from $P_{t(u)}$ to $P_{o(u)}$ is the left
multiplication by the path $u$ with origin $o(u)$ and terminus
$t(u)$. In fact, the complex $P^{\bullet}_{\lambda, d}$ can be
illustrated as follows
$$\xymatrix@!=1pc{
P_1^d \ar [rrd]_-{P(w_5)\mathbf{J}_{\lambda,d}} \ar [rr]^-{P(w'_5)
\mathbf{I}_d} && P_{m}^d \ar [rr]^-{P(w'_4) \mathbf{I}_d} &&
P_{m+1}^d \ar [rr]^-{P(w'_3) \mathbf{I}_d} && P_{2m}^d \ar
[rr]^-{P(w'_2) \mathbf{I}_d} && P_{2m+1}^d
\ar [rr]^-{P(w'_1)\mathbf{I}_d} && P_{3m}^d \\
&& P_{2}^d \ar [rr]^-{P(w_4)\mathbf{I}_d} && P_{m+1}^d \ar
[rr]_-{P(w_3) \mathbf{I}_d} && P_{2m}^d \ar [rr]_-{P(w_2)
\mathbf{I}_d} && P_{3m-1}^d \ar [rru]_-{P(w_1)\mathbf{I}_d} }$$
where $P_1^d$ lies in the $0$-th component of $P^{\bullet}_{\lambda,
d}$.

It is elementary to show that $\End_{C^b(B)}(P^{\bullet}_{\lambda,
d})$ is local, i.e., the complex $P^{\bullet}_{\lambda, d}$ is
indecomposable, for all $\lambda \in k$ and $d \geq 1$. Indeed, if
$f^{\bullet}=(f^i)\in \End_{C^b(B)}(P^{\bullet}_{\lambda, d})$ then
$f^i=0$ unless $0\leq i\leq 5$. As a $k$-vector space,
$$\Hom_B(P_i, P_j) \cong e_jBe_i \cong
\left\{\begin{array}{ll} k, & \mbox{if } i \leq j < i+m ;
\\ 0, & \mbox{otherwise. } \end{array}\right.$$
Thus each $f^i$ can be expressed as a matrix and by the construction
of $P^{\bullet}_{\lambda, d}$, $f^2$ and $f^3$ can be written as the
same block matrix of form
$$\left(\begin{array}{cc} f_{11} & f_{12} \\ f_{21} & f_{22} \\
\end{array}\right), \, f_{ij} \in k^{d \times d}.$$
Since $m>2$, we have
$$f^1 = \left(\begin{array}{ll} f^1_{11} & f^1_{12} \\ 0 & f^1_{22}
\end{array}\right), \;
f^4=\left(\begin{array}{ll} f^4_{11} & 0 \\ f^4_{21} & f^4_{22}
\end{array}\right), \, f^h_{ij} \in k^{d \times d}.$$ The commutativity of
cochain map forces $$f^1 = f^2 = f^3 = f^4 = \left(\begin{array}{cc}
f_{11} & 0 \\ 0 & f_{22} \end{array}\right).$$ Furthermore, $f^0,
f^5 \in k^{d \times d}$ satisfy
$$\left(\begin{array}{c} \mathbf{I}_d \\ \mathbf{J}_{\lambda, d} \end{array}\right)f^0
=
\left(\begin{array}{cc} f_{11} & 0 \\ 0 & f_{22} \end{array}\right)
\left(\begin{array}{c} \mathbf{I}_d \\ \mathbf{J}_{\lambda, d}
\end{array}\right)$$ and $$f^5 \left(\begin{array}{cc} \mathbf{I}_d &
\mathbf{I}_d \end{array}\right) = \left(\begin{array}{cc}
\mathbf{I}_d & \mathbf{I}_d
\end{array}\right) \left(\begin{array}{cc} f_{11} & 0
\\ 0 & f_{22} \end{array}\right). $$
Therefore, $f^0 = f_{11} = f_{22} = f^5$ and $\mathbf{J}_{\lambda,
d} f^0 = f^0 \mathbf{J}_{\lambda, d} $, and thus $f^0$ is of the
form
$$\left(\begin{array}{ccccc} x_1 & x_2 & \cdots & x_{d-1} & x_d
\\ 0 & x_1 & \cdots & x_{d-2} & x_{d-1}\\ \vdots & \vdots & \ddots & \vdots & \vdots
\\ 0 & 0 & \cdots & x_1 & x_2 \\ 0 & 0 & \cdots & 0 & x_1
\end{array}\right), \; x_i \in k. $$ Hence $\End_{C^b(B)}(P^{\bullet}_{\lambda, d})$ is local.
Moreover, the complexes $\{P^{\bullet}_{\lambda, d} \; | \; \lambda
\in k, d \geq 1\}$ are pairwise different up to shift and
isomorphism by a similar argument on the morphisms between these
$P^{\bullet}_{\lambda, d}$'s.

Now it suffices to show that $\hr(P^{\bullet}_{\lambda, d}) =
\hr(P^{\bullet}_{\lambda', d'})$ and $\hl(P^{\bullet}_{\lambda, d})
= \hl(P^{\bullet}_{\lambda', d'})$ if and only if $d=d'$, which
implies $B$ is strongly derived unbounded and ${\rm gl.hl} B =
\infty$. Indeed, it is clear that $H^i(P^{\bullet}_{\lambda, d})$ is
independent of $\lambda$ except $i=0, 1$. Moreover,
$H^0(P^{\bullet}_{\lambda, d}) = 0$ and $\dim
H^1(P^{\bullet}_{\lambda, d})$ is independent of $\lambda$ since
$\delta^0$ is injective. Hence, $P^{\bullet}_{\lambda, d}$'s are of
the same cohomological range and cohomological length for a fixed
$d$. Conversely, we have $\hw(P^{\bullet}_{\lambda, d}) = 5$ due to
$H^1(P^{\bullet}_{\lambda, d}) \neq 0 \neq H^5(P^{\bullet}_{\lambda,
d})$ and $\hl(P^{\bullet}_{\lambda, d}) = d \cdot
\hl(P^{\bullet}_{\lambda, 1})$. Thus, $P^{\bullet}_{\lambda, d}$'s
are of distinct cohomological ranges and cohomological lengths for
different $d$.
\end{proof}

\begin{lemma}
\label{lemma-sdu-end} If a finite spectroid $A$ is not strongly
derived unbounded (resp. $A$ is of finite global cohomological
length) then the endomorphism algebra $A(a, a)$ is isomorphic to
either $k$ or $k[x]/(x^2)$ for all $a \in A$.
\end{lemma}

\begin{proof} If $A$ is not strongly derived unbounded (resp.
$A$ is of finite global cohomological length) then $A$ is
representation-finite. Thus for any $a \in A$, $A(a,a)$ is a
uniserial local algebra, and then $A(a,a) \cong k$ or $A(a,a) \cong
k[x]/(x^m)$ with $m \geq 2$. Note that the functor $F : A_n^m
\rightarrow A$ given by $F(i)=a$ and $F(\alpha_j)=x$ is a cleaving
functor. If $m \geq 3$ then, by Lemma \ref{lemma-sdu-a}, $A_{3m}^m$
is strongly derived unbounded and ${\rm gl.hl} A_{3m}^m = \infty$.
It follows from Proposition \ref{prop-cf-sdu} that $A$ is strongly
derived unbounded and ${\rm gl.hl} A = \infty$, which is a
contradiction.
\end{proof}

Now we can prove Theorem II for all algebras.

\begin{theorem}
\label{thm-dBT2}A finite spectroid is either derived discrete or
strongly derived unbounded.
\end{theorem}

\begin{proof} Assume that a finite spectroid $A$ is not
strongly derived unbounded. Then $A$ is representation-finite. It
follows from Lemma \ref{lemma-sdu-end} that the endomorphism algebra
$A(a, a)$ is isomorphic to either $k$ or $k[x]/(x^2)$ for all $a \in
A$. Thus $A$ does not contain Riedtmann contours, and hence it is
standard \cite[Section 9]{BGRS85}.

If $A$ is simply connected then $A$ is derived discrete by Lemma
\ref{lemma-dbt2-sc}. If $A$ is not simply connected then $A$ admits
a Galois covering $\pi: \tilde{A} \rightarrow A$ with non-trivial
free Galois group $G$ such that $\tilde{A}$ is a simply connected
locally bounded spectroid \cite{BG83}, hence the filtered union of
its connected convex finite full subspectroids \cite{BG83, Ga81}.
Any connected convex finite full subspectroid $B$ of $\tilde{A}$ is
simply connected, thus $\mbox{\rm gl.dim } B < \infty$. Note that
the composition of the embedding functor $B \hookrightarrow
\tilde{A}$ and the covering functor $\pi$ is a cleaving functor. By
Proposition \ref{prop-cf-sdu}, $B$ is not strongly derived
unbounded. It follows from Lemma \ref{lemma-dbt2-sc} that $B$ is
piecewise hereditary of Dynkin type. By the same argument as that in
the proof of \cite[Lemma 4.4]{Vo01}, we obtain $B$ is piecewise
hereditary of type $\mathbb{A}$. Thus $\tilde{A}$ admits a
presentation given by a gentle quiver $(Q,I)$ (Ref.
\cite[Theorem]{AH81}), and so does $A$. Therefore, $A$ is derived
discrete by Lemma \ref{lemma-dbt2-gentle}.
\end{proof}

Next we show that derived discrete algebras can be characterized as
the algebras of finite global cohomological length. Moreover, we
summarize in the following proposition all previous results on
global finiteness of the homological invariants introduced in this
paper.

\begin{proposition}
\label{prop-dd-glhl} Let $A$ be a finite spectroid. The following
assertions hold:

{\rm (1)} $\mbox{\rm gl.hl} A < \infty$ if and only if $A$ is
derived discrete;

{\rm (2)} $\mbox{\rm gl.hw} A < \infty$ if and only if $A$ is
piecewise hereditary;

{\rm (3)} $\mbox{\rm gl.hr} A < \infty$ if and only if $A$ is
piecewise hereditary of Dynkin type.
\end{proposition}

\begin{proof}
If $A$ is derived discrete then by Vossieck's classification of
derived discrete algebras \cite[Theorem]{Vo01}, $A$ is either
piecewise hereditary of Dynkin type or derived equivalent to some
gentle algebras without generalized bands. In the case $A$ is
piecewise hereditary of Dynkin type, by Corollary \ref{coro-equi-h},
we have $\mbox{gl.hl} A < \infty$. In the other case, by Lemma
\ref{lemma-dbt2-gentle}, we have $\mbox{gl.hl} A < \infty$.

Conversely, it is enough to repeat the proof of Theorem
\ref{thm-dBT2} and replace the phrase ``not strongly derived
unbounded'' with ``of finite global cohomological length''.

The statements (2) and (3) are actually Corollary
\ref{coro-char-phere} and Theorem \ref{theorem-dbt1} respectively.
\end{proof}

\begin{remark} \label{Remark-AnotherProofOfTheoremI} {\rm
By Proposition \ref{prop-dd-glhl}, we know piecewise hereditary
algebras and derived discrete algebras can be characterized as the
algebras of finite global cohomological width and the algebras of
finite global cohomological length respectively, which provides
another proof of the first Brauer-Thrall type theorem for derived
category. Indeed, an algebra $A$ satisfies ${\rm gl.hr} A < \infty$
if and only if both ${\rm gl.hw} A < \infty$ and ${\rm gl.hl} A <
\infty$, if and only if $A$ is both piecewise hereditary and derived
discrete, i.e., piecewise hereditary of Dynkin type. }\end{remark}

We conclude this paper with a question. In \cite{Bo13}, Bongartz
proved that for a finite-dimensional algebra $A$ over an
algebraically closed field $k$, there are no gaps in the sequence of
dimensions of finite-dimensional indecomposable $A$-modules. More
precisely, if there is an indecomposable $A$-module of dimension $n
\geq 2$ then there is also one of dimension $n-1$. It is natural to
consider the derived version of the above Bongartz's theorem and ask
whether there are no gaps in the sequence of cohomological ranges of
indecomposable objects in $D^b(A)$.

\begin{question} Is there an indecomposable object in $D^b(A)$ of cohomological
range $r-1$ if there is one of cohomological range $r \geq 2$?
\end{question}

\medskip

\noindent {\footnotesize {\bf ACKNOWLEDGEMENT.} The authors thank
Xiao He and Yongyun Qin for many helpful discussions on this topic.
They are supported by the National Natural Science Foundation of
China (Grant No. 11171325 and 11571341).}

\end{document}